\documentclass[11pt]{article}
\usepackage{caption}
\usepackage{epsf,epsfig,amsfonts,amsgen,amsmath,amstext,amsbsy,amsopn,amsthm
}
\usepackage{ebezier,eepic}
\usepackage{color}
\usepackage{multirow}
\usepackage{mathrsfs}
\usepackage{tikz}
\usepackage{enumerate}

\newtheorem{dfn}{Definition}[section]
\newtheorem{thm}[dfn]{Theorem}
\newtheorem{lem}[dfn]{Lemma}
\newtheorem{prop}[dfn]{Proposition}

\newtheorem{constr}[dfn]{Construction}
\newtheorem{prob}[dfn]{Problem}
\newtheorem{Dfn}{Definition}
\newtheorem{claim}[Dfn]{Claim}

\newcommand{\dB}{{\mathcal{B}}}
\newcommand{\dC}{{\mathcal{C}}}
\newcommand{\dG}{{\mathcal{G}}}
\newcommand{\dT}{{\mathcal{T}}}

\newcommand{\1}{{\uppercase\expandafter{\romannumeral1}}}
\newcommand{\2}{{\uppercase\expandafter{\romannumeral2}}}
\newcommand{\3}{{\uppercase\expandafter{\romannumeral3}}}
\newcommand{\4}{{\uppercase\expandafter{\romannumeral4}}}

\begin{document}

\title{Counting critical subgraphs in $k$-critical graphs}

\author{
Jie Ma\thanks{School of Mathematical Sciences, University of Science and Technology of China,
Hefei, Anhui 230026, China. Email: jiema@ustc.edu.cn. Partially supported by NSFC grant 11622110.}
~~~~~
Tianchi Yang\thanks{School of Mathematical Sciences, University of Science and Technology of China,
Hefei, Anhui 230026, China. Email: ytc@mail.ustc.edu.cn.}
}
\maketitle

\begin{abstract}
Gallai asked in 1984 if any $k$-critical graph on $n$ vertices contains at least $n$ distinct $(k-1)$-critical subgraphs.
The answer is trivial for $k\leq 3$.
Improving a result of Stiebitz \cite{S87}, Abbott and Zhou \cite{AZ} proved in 1995 that
for all $k\geq 4$, such graph contains $\Omega(n^{1/(k-1)})$ distinct $(k-1)$-critical subgraphs.
Since then no progress had been made until very recently, Hare \cite{H19} resolved the case $k=4$ by showing that any $4$-critical graph on $n$ vertices contains at least $(8n-29)/3$ odd cycles.

In this paper, we mainly focus on 4-critical graphs and develop some novel tools for counting cycles of specified parity.
Our main result shows that any $4$-critical graph on $n$ vertices contains $\Omega(n^2)$ odd cycles, which is tight up to a constant factor by infinite many graphs.
As a crucial step, we prove the same bound for 3-connected non-bipartite graphs, which may be of independent interest.
Using the tools, we also give a very short proof for the case $k=4$.
Moreover, we improve the longstanding lower bound of Abbott and Zhou to $\Omega(n^{1/(k-2)})$ for the general case $k\geq 5$.
We will also discuss some related problems on $k$-critical graphs in the final section.
\end{abstract}

\section{Introduction}
In this paper, all graphs referred are {\it simple} graphs, unless otherwise specified.
The {\it chromatic number} $\chi(G)$ of a graph $G$ is the minimum number of colors to be assigned to its vertices so that no adjacent vertices receive the same color.
A graph $G$ is called {\it $k$-critical} if it has chromatic number $k$ but every proper subgraph has chromatic number less than $k$.
Note that 3-critical graphs are all odd cycles.

In 1984, Gallai asked the following problem (see Problem 5.9 of \cite{JT} or the discussion in \cite{S87}).
\begin{prob}[Gallai]\label{prob:Gallai}
If $G$ is a $k$-critical graph on $n$ vertices, is it true that $G$ contains $n$ distinct $(k-1)$-critical subgraphs?
\end{prob}

\noindent
This problem is trivial for $k\leq 3$.
From now on, we will assume $k\geq 4$.
For convenience, for each $s\geq 3$ we denote $f_s(G)$ by the number of distinct $s$-critical subgraphs in a graph $G$.
For $s=3$, we will simply write $f(G)$ instead.
Let $G$ be an $n$-vertex $k$-critical graph.
%Using a result of Toft \cite{T74},
Stiebitz \cite{S87} first proved that $f_{k-1}(G)\geq \log_2n.$
This was improved by Abbott and Zhou \cite{AZ} to $$f_{k-1}(G)\geq ((k-1)!n)^{\frac{1}{k-1}}$$ in 1995 and there has been no further improvement for general $k$ since then.
Very recently, Hare \cite{H19} answered Gallai's problem in the case $k=4$
by showing that every $4$-critical graph on $n$ vertices contains at least $\frac83n-\frac{29}3$ odd cycles.

Our first result improves the general bound of Abbott and Zhou \cite{AZ} for every $k\geq 4$.
\begin{thm} \label{Thm:general k}
For $k\geq 4$, every $k$-critical graph $G$ on $n$ vertices satisfies
$\binom{f_{k-1}(G)}{k-2}\geq e(G)$.
Thus
\begin{equation*}
f_{k-1}(G)\geq ((k-1)!n/2)^{\frac{1}{k-2}}.
\end{equation*}
\end{thm}

\begin{proof}
For each $e\in E(G)$, $G-e$ has a proper $(k-1)$-coloring, say with color classes $A_1,...,A_{k-1}$, where $V(e)\subseteq A_1$.
For each $2\leq i\leq k-1$, we see that $G-A_i$ has chromatic number $k-1$ and thus contains a $(k-1)$-critical subgraph $G_i^e$.
It is also clear that $e\in E(G_i^e)$.
Let $L(e)=\{G_2^e,...,G_{k-1}^e\}$.
Note that each graph in $L(e)$ is $(k-1)$-critical and contains $e$.
We claim that for any $f\in E(G-e)$ there is at least one subgraph in $L(e)$ not containing $f$.
To see this, we may assume $f=uv$ with $u\in A_i$ and $v\in A_j$ for some $1\leq i<j\leq k-1$, implying that $f\notin E(G_j^e)$.
This claim shows that $L(e)$ are distinct for all edges $e$ in $G$ and so $\binom{f_{k-1}(G)}{k-2}\geq e(G)$.
Since $e(G)\geq (k-1)n/2$, this further implies $f_{k-1}(G)\geq ((k-1)!n/2)^{\frac{1}{k-2}}$.
\end{proof}

Having this, our main focus is devoted to the case of 4-critical graphs.
Among others, we prove a tight bound on the number of odd cycles in 4-critical graphs.
This in fact is provided in a stronger form, which reveals a relation between the numbers of odd cycles and edges.
To state, we begin by introducing a parameter which will play an important role in the proofs:
for any graph $G$, let $$t(G)=|E(G)|-|V(G)|+1.$$
Note that if $G$ is 2-connected, then any ear-decomposition of $G$ has exactly $t(G)$ ears;
also, for 4-critical graph $G$, since every vertex has degree at least 3, we have $t(G)\geq |E(G)|/3\geq |V(G)|/2$.

\begin{thm}\label{Thm:4-cri-str}
If $G$ is a 4-critical graph on $n$ vertices and $m$ edges, then $f(G)\geq 0.02t^2(G)$.
Thus $$f(G)\geq \Omega(m^2)\geq \Omega(n^2).$$
\end{thm}

\noindent We remark that this is tight up to the constant factor.
To see this, by an $n$-vertex {\it $d$-wheel} $W(n, d)$ we denote the graph obtained from a cycle $C_{n-d}$ and a clique $K_d$ by joining each vertex of $C_{n-d}$ to each vertex of $K_d$.
It is {\it odd} if $n-d$ is odd and {\it even} otherwise.
For simplicity, we just call a $1$-wheel as a wheel.
Now we observe that the odd wheel $W=W(n, 1)$ is 4-critical and has $\binom{n-1}{2}+1$ odd cycles;
it also has $O(|E(W)|^2)$ and $O(t^2(W))$ odd cycles.

As an intermediate step and a result of independent interest, we prove a similar bound for 3-connected non-bipartite graphs as following.

\begin{thm}\label{thm:G=3-con}
If $G$ is a 3-connected non-bipartite graph, then $f(G)\geq 0.02t^2(G)$.
\end{thm}

We also give a new proof to the case $k=4$ of Problem \ref{prob:Gallai}.
It is significantly shorter than the one given in \cite{H19} -- by adding all rudimentary lemmas, it is about 2-page long.

\begin{thm}\label{Thm:4-cri-weak}
If $G$ is a 4-critical graph on $n$ vertices, then $f(G)\geq 2t(G)-2=2e(G)-2n$.
In particular $f(G)\geq n$, where the unique 4-critical graph achieving the equality is $K_4$ when $n=4$.
\end{thm}

We shall explain in the final section that all results on 4-critical graphs here also hold for $k$-critical graphs for all $k\geq 4$.

The rest of the paper is organized as following.
In Section 2, we give notations and collect basic lemmas for further use.
We then prove some lemmas for 3-connected non-bipartite signed graphs in Section 3
and using these, we give a short proof of Theorem \ref{Thm:4-cri-weak} in Section 4.
In Section 5, we prove two technical lemmas as tools for counting cycles of each parity.
In Section 6, we complete the proof of Theorem \ref{thm:G=3-con} by detouring to signed graphs.
In Section 7, we prove Theorem \ref{Thm:4-cri-str}.
The final section contains some concluding remarks and related problems.
We remark that we do not attempt to optimize the constant factors in our results,
preferring rather to provide a simpler presentation.
%%%%%%%%%%%%%%%%%%%%%%%%%%%%%%%%%%%%%%%%%%%%%%%%%%%%%%%%%

\section{Preliminaries}
The following structure lemma on $k$-critical graphs was first proved by Dirac \cite{G53,G64},
a detailed proof of which also can be found in \cite{ST} (see its Lemma 3.2).
\begin{lem}[\cite{G53,G64}]\label{lem:2-cut}
Let $k\geq 4$ be an integer, $G$ be a $k$-critical graph and $\{u,v\}$ be a 2-cut of $G$. Then
$uv\notin E(G)$ and there are unique proper induced subgraphs $G_1, G_2$ of $G$ such that
\begin{itemize}
\item[(a)] $G=G_1\cup G_2$ and $V(G_1)\cap V(G_2)=\{u,v\}$,
%\item[(b)] $G_1-\{u,v\}$ and $G_2-\{u,v\}$ are the only two components of $G-\{u,v\}$,
\item[(b)] $u$ and $v$ have no common neighbor in $G_2$, and
\item[(c)] both $G_1+uv$ and $G_2/\{u,v\}$ are $k$-critical.\footnote{The graph $G_2/\{u,v\}$ is obtained from $G_2$ by contracting $u$ and $v$ into a new vertex.}
\end{itemize}
\end{lem}

Answering a long-standing conjecture of Ore from 1967 on the number of edges in 4-critical graphs,
Kostochka and Yancey \cite{KY} proved the following tight result.

\begin{thm}[\cite{KY}] \label{Thm:KY}
If $G$ is a 4-critical graph, then $e(G)\geq \frac{5}{3}|V(G)|-\frac{2}{3}$.
\end{thm}

Given a subgraph $F$ in a graph $G$, by $G-F$ we denote the subgraph obtained from $G$ by deleting all vertices in $F$.
We say a cycle $C$ is {\it non-separating} in $G$ if $G-C$ is connected.
In 1980 Krusenstjerna-Hafstr{\o}m and Toft proved the following theorem (Theorems 4 and 5 in \cite{KHT}).

\begin{thm}[\cite{KHT}]\label{thm:KT}
Let $G$ be a graph which is either 4-critical or 3-connected and let $F$ be a connected subgraph of $G$
such that $G-F$ contains an odd cycle.
Then $G$ contains a non-separating induced odd cycle $C$ such that $V(C)\cap V(F)=\emptyset$.
\end{thm}

A path with end-vertices $x$ and $y$ is called an {\it $(x,y)$-path}.
Let $G$ be a given graph (not necessarily connected).
A vertex $v\in V(G)$ is called a {\it cut-vertex} of $G$ if $G-v$ has more components than $G$.
A {\it block} $B$ of $G$ is a maximal connected subgraph of $G$ such that there exists no cut-vertex of $B$.
So a block is either an isolated vertex, an edge or a 2-connected graph.

\begin{lem} \label{lem path}
For any two distinct vertices $x,y$ in a block $B$, there are at least $t(B)+1$ distinct $(x,y)$-paths in $B$.
\end{lem}

\begin{proof}
If $B$ is an edge $xy$, then this holds trivially. So we may assume that $B$ is 2-connected.
Let $t:=t(B)$ and $C$ be any cycle containing $x$ and $y$.
By the standard ear decomposition of a 2-connected graph, there exist $t-1$ paths $P_1,P_2,...,P_{t-1}$ in $B$
such that $B_i:=C\cup (\cup_{j=1}^i P_j)$ is 2-connected for each $0\leq i\leq t-1$, where $B_0=C$ and $B_{t-1}=B$.
For each $1\leq i\leq t-1$, let $a_i$ and $b_i$ be the end-vertices of $P_i$.
As $B_{i-1}$ is 2-connected, there exist two disjoint paths from $\{a_i,b_i\}$ to $\{x,y\}$ in $B_{i-1}$.
This gives an $(x,y)$-path in $B_i$ containing the path $P_i$.
Together with the two $(x,y)$-paths in $C$, we get at least $t+1$ distinct $(x,y)$-paths in $B$.
\end{proof}

Let $\dB$ be the set of blocks in a graph $G$ and $\dC$ be the set of cut-vertices of $G$.
The \emph{block structure} of $G$ is the bipartite graph with bipartition $(\dB,\dC)$, where $c\in \dC$ is adjacent to $B_i\in \dB$ if and only if
$c\in V (B_i)$. Note that the block structure of any connected graph is a tree.
An \emph{end-block} in $G$ is a block containing at most one cut-vertex of $G$.

\begin{prop}\label{prop:t}
Let $G$ be a connected graph.
Then $t(G)=\sum_{B\in \dB} t(B)$.
\end{prop}

\begin{proof}
This can be proved easily by induction on the number of blocks using the block structure.
\end{proof}

A {\it signed graph} is a graph $G$ associated with a function $p: E(G) \to \{0,1\}$.
For $e\in E(G)$, we refer $p(e)$ as the {\it parity} of $e$.
The parity of a path or a cycle $C$ in $G$ is the parity of the sum of the parities of all edges in $E(C)$,
and we say $C$ is {\it even} if its parity is 0 and {\it odd} otherwise.
A signed graph is {\it bipartite} if every cycle is even and {\it non-bipartite} otherwise.
In this paper we view every graph as a signed graph by assigning 1 to every edge.
The following property can be derived promptly.

\begin{prop}\label{prop:bip}
A signed graph $(G,p)$ is bipartite if and only if there exists a bipartition $V(G)=A\cup B$ such that each $e\in E(A,B)$ is odd and each $e\in E(G)\backslash E(A,B)$ is even.
\end{prop}

We also need a lemma proved by Kawarabayashi, Reed and Lee (see Lemma 2.1 in \cite{KRL}).

\begin{lem}[\cite{KRL}]\label{lem_nonseparate}
If $s$ is a vertex in a 3-connected signed graph $G$ such that $G-s$ is not bipartite,
then there is a non-separating induced odd cycle $C$ in $G$ with $s\notin V(C)$.
\end{lem}

Throughout the rest of this paper, a set of edges is called {\it independent} if their vertices are all disjoint.
For any integer $k\geq 1$, we write $[k]$ as $\{1,2,...,k\}$.

\section{Lemmas on 3-connected non-bipartite signed graphs}\label{sec:3-conn}
Throughout this section, let $G$ be a 3-connected non-bipartite signed graph.
By Lemma \ref{lem_nonseparate}, there exists an induced odd cycle $C$ in $G$ such that $G-C$ is connected.
Fix such a cycle $C$ and let $H=G-C$, $t=t(H)$ and $m=|E(C,H)|$.
Then it is straightforward to see that $t(G)=t+m$.

A pair of edges $xa, yb\in E(C,H)$ with $x,y\in V(C)$ and $a,b\in V(H)$ is called {\it good} if $x\neq y$.
Given such a pair $\{xa,yb\}$, we call any $(a,b)$-path contained in $H$ a {\it good path}.
It is easy to see that any good $(a,b)$-path in $H$ can be uniquely extended to an odd cycle in $G$ by adding $xa,yb$ and one of the two $(x,y)$-paths in $C$.
Such an odd cycle will be called \emph{basic} in $G$ for the good pair $\{xa,yb\}$.

\begin{lem}\label{lem: 2-connected H}
If $H$ is 2-connected, then there are at least $(t+1)m$ distinct basic cycles in $G$.
\end{lem}
\begin{proof}
Clearly we have $|C|\geq 3$ and $|V(H)|\geq 3$.
Since $G$ is 3-connected, there are 3 independent edges $x_ia_i\in E(C,H)$ with $x_i\in V(C)$ and $a_i\in V(H)$ for all $i\in [3]$.
By Lemma \ref{lem path}, for different $i,j\in [3]$, we get at least $t+1$ distinct $(a_i,a_j)$-paths in $H$.
This gives at least $3(t+1)$ distinct basic cycles in $G$ using exactly two of $\{x_1a_1, x_2a_2, x_3a_3\}$.
For any $yb\in E(C,H)$ other than $\{x_ia_i\}$,
there is at least one edge (say $x_1a_1$) in $\{x_ia_i\}$ independent of $yb$.
Using Lemma \ref{lem path}, similarly one can find at least $t+1$ distinct basic cycles using $yb$ and $x_1a_1$.
Together we see at least $3(t+1)+(m-3)(t+1)=(t+1)m$ distinct basic cycles in $G$.
\end{proof}

Let $\dB$ be the set of blocks in $H$ and $\dC$ be the set of cut-vertices in $H$.
For $a,b\in V(H)$, by $\mathcal{P}_{a,b}$ we denote the shortest path $B_{j_1}c_1B_{j_2}c_2...c_{\ell-1}B_{j_\ell}$ in the block structure $(\dB,\dC)$ of $H$
satisfying that $a\in V(B_{j_1})$ and $b\in V(B_{j_\ell})$, where $B_i\in \dB$ and $c_j\in \dC$.
%For each block $B_j\in \dB$, let $t_j=t(B_j)$.

\begin{lem}\label{lem:block tree}
Let $a,b\in V(H)$ be two distinct vertices.
Then there are at least $\prod_{B\in \mathcal{P}_{a,b}\cap \dB}(t(B)+1)\geq \left(\sum_{B\in \mathcal{P}_{a,b}\cap \dB} t(B)\right)+1$ distinct $(a,b)$-paths in $H$.
\end{lem}

\begin{proof}
Let $B_1c_1B_2c_2...c_{\ell-1}B_\ell$ be the path $\mathcal{P}_{a,b}$, where $a\in V(B_1)$ and $b\in V(B_\ell)$.
Let $c_0=a$ and $c_\ell=b$.
By Lemma \ref{lem path}, there are at least $t(B_i)+1$ distinct $(c_{i-1},c_i)$-paths in $B_i$ for each $1\leq i\leq \ell$, implying this lemma.
\end{proof}

In the rest of this section, we assume that $H$ is connected but not 2-connected.
For each end-block $B_i$ in $H$, we define the unique cut-vertex of $H$ in $B_i$ to be $c_i$.
We now define a good pair of edges $\{e_i,f_i\}$ in $E(C,B_i-c_i)$, called {\it staple edges} of the end-block $B_i$, as follows.
If $B_i$ is an edge say $a_ic_i$, as $a_i$ has at least two neighbors $x_i,y_i\in V(C)$,
let $e_i=x_ia_i$ and $f_i=y_ia_i$.
Otherwise $B_i$ is 2-connected with $|V(B_i)|\geq 3$.
There are 3 disjoint paths from $B_i$ to $C$ in $G$ (as $G$ is 3-connected)
at most one of which uses the cut-vertex $c_i$,
so the other two paths must be two independent edges say $e_i=x_ia_i$ and $f_i=y_ib_i$ in $E(C,B_i-c_i)$.
%We note that there are exactly $2k$ staple edges.
%and for each end-block $B_i$, the staple edges $e_i, f_i$ share a common vertex $w$ if and only if $B_i=a_iu_i$ and $w=a_i$.

\begin{lem}\label{lem:not 2 connected}
Let $k$ be the number of end-blocks in $H$. If $k\geq 2$, then there are at least $(m-k)(t+k)+\lceil \frac{k}{2}\rceil$ basic cycles in $G$.
\end{lem}

\begin{proof}
Let $B_1,B_2,...,B_k$ be all end-blocks in $H$.
Let $uv$ be a non-staple edge in $E(C,H)$ with $v\in V(H)$.
For each $B_i$, at least one of $e_i, f_i$ has an end-vertex in $V(C)-u$;
let $e_i=x_ia_i$ be such an edge with $a_i\in V(B_i)-c_i$ and thus $\{uv, x_ia_i\}$ is a good pair.
Since the block structure of $H$ is a tree,
the union of the $k$ paths $\mathcal{P}_{v,a_i}$ over $i\in [k]$ contains all blocks in $\dB$.
By Lemma \ref{lem:block tree} and Proposition \ref{prop:t}, there are at least $(\sum_{B\in \dB} t(B))+k=t+k$ distinct $(v,a_i)$-paths for all $i\in [k]$.
This gives $t+k$ basic cycles in $G$ using $uv$ and exactly one staple edge.
Since there are $m-2k$ non-staple edges in $E(C,H)$,
we have at least $(m-2k)(t+k)$ distinct basic cycles in $G$ using exactly one staple edge.

We then consider basic cycles with two staple edges.
For end-blocks $B_i,B_j$, we can always pair the four staple edges $e_i,f_i,e_j,f_j$ into two good pairs $\mathcal{A}_\ell$ for $\ell\in [2]$ with $|\mathcal{A}_\ell\cap \{e_i,f_i\}|=1$.
Thus each of the $2k$ staple edges (say $e_1$) appears in $k$ good pairs $\{e_1,g_j\}$ for $j\in [k]$, where $g_j$ is a staple edge of $B_j$.
Similarly as above, each staple edge is contained in at least $t+k$ basic cycles using two staple edges.
By double-counting, this gives at least $k(t+k)$ basic cycles using two staple edges.

Now consider the staple edges $e_i,f_i$ of each $B_i$.
As $G$ is 3-connected, there exists $g\in E(C,H)$ independent of $e_i,f_i$.
Thus $\{g,e_i\}$ and $\{g,f_i\}$ both are good pairs.
Note that such edge $g$ may be a staple edge or not, and we have only considered one good pair for $g$ in the above counting.
By double-counting (as $g$ can be a staple edge), we can get $\lceil\frac{k}{2}\rceil$ more good pairs,
which lead to $\lceil\frac{k}{2}\rceil$ more distinct basic cycles in $G$.
This lemma follows by adding all above basic cycles up.
\end{proof}

We make two remarks: (1) The odd cycle $C$ is not a basic cycle.
(2) Each basic cycle corresponds to a unique even cycle.
So Lemmas \ref{lem: 2-connected H} and \ref{lem:not 2 connected} give the same number of distinct even cycles in $G$.

\section{A short proof to Gallai's problem when $k=4$}

\begin{lem}\label{lem:3-con_simple}
Every 3-connected non-bipartite graph $G$ contains at least $2t(G)-2$ distinct odd cycles.
\end{lem}

\begin{proof}
%We may view $G$ as a signed graph with every edge being odd.
Following the notations in Section \ref{sec:3-conn},
let $C$ be an induced odd cycle in $G$ such that $G-C$ is connected.
Let $H=G-C$, $t=t(H)$ and $m=|E(C,H)|$.
Then we have $t(G)=t+m$.
If $H$ is 2-connected, then $t\geq 1$ and $m\geq 3$.
Since $(t+1)m-(2t(G)-2)=(t-1)(m-2)\geq 0$,
by Lemma \ref{lem: 2-connected H}, $G$ contains at least $(t+1)m\geq 2t(G)-2$ odd cycles.
So $H$ is not 2-connected.
Let $k$ be the number of end-blocks in $H$.

If $k\geq 2$, then $m\geq 2k\geq k+2$ and thus $(m-k)(t+k)+2=((m-k-2)+2)((t+k-2)+2)+2\ge 2(m-k-2)+2(t+k-2)+6=2t(G)-2$.
By Lemma \ref{lem:not 2 connected} (plus the cycle $C$), $G$ contains at least $(m-k)(t+k)+2\geq 2t(G)-2$ odd cycles.
It remains to consider $k=1$, that is, $H$ is an isolated vertex or an edge.
If $H$ is a vertex, then every two edges in $E(C,H)$ form a good pair.
If $H$ is an edge $ab$, then any non-good pair in $E(C,H)$ must be $\{ax,bx\}$ for some $x\in V(C)$, which also defines a triangle $abx$.
Hence in either case, it holds that $t=0$, $t(G)=m$ and any pair in $E(C,H)$ contributes a distinct odd cycle in $G$.
Adding the cycle $C$, there are at least $\binom{m}{2}+1=\frac{1}{2}t(G)(t(G)-1)+1\geq 2t(G)-2$ odd cycles in $G$,
where the inequality holds as $t(G)\geq |V(G)|/2+1\geq 2$.
This completes the proof.
\end{proof}

Now we are ready to prove Theorem \ref{Thm:4-cri-weak}.

\medskip

{\noindent \bf Proof of Theorem \ref{Thm:4-cri-weak}.}
Let $G$ be a 4-critical graph on $n$ vertices.
We prove $f(G)\geq 2t(G)-2$ by induction on $n$. It is clear that if $n=4$, then $G=K_4$ has exactly 4 odd cycles.
So we may assume that this holds for all 4-critical graphs with at most $n-1$ vertices.

Clearly $G$ is 2-connected and non-bipartite.
If $G$ is 3-connected, then Lemma \ref{lem:3-con_simple} implies $f(G)\geq 2t(G)-2$.
So we may assume that there exists a 2-cut $\{u,v\}$ in $G$.
By Lemma \ref{lem:2-cut}, $uv\notin E(G)$ and there exist induced subgraphs $G_1$ and $G_2$ of $G$
such that $G=G_1\cup G_2$, $V(G_1)\cap V(G_2)=\{u,v\}$, and $H_1:=G_1+uv$ and $H_2:=G_2/\{u,v\}$ are 4-critical.
Also $u,v$ have no common neighbor in $G_2$, so $e(H_2)=e(G_2)$, from which we can derive that $t(H_1)+t(H_2)=t(G)+1$.

We claim that both $G_1$ and $G_2$ contain two $(u,v)$-paths of different parities.
Since $H_1$ is 4-critical and thus 2-connected, there exist an odd cycle $C$ not containing $u$ and two disjoint paths from $u,v$ to $C$ in $H_1$ (also in $G_1$).
Then we can easily get two $(u,v)$-paths of different parities in $G_1$.
Similarly, $H_2$ has an odd cycle $D$ avoiding the new vertex contracted from $\{u,v\}$.
There are two disjoint paths from $u,v$ to $D$ in the 2-connected $G$.
Clearly these paths are also contained in $G_2$. Thus we can get two $(u,v)$-paths of different parities in $G_2$.

Suppose that the numbers of $(u,v)$-paths of even length in $G_1,G_2$ are $a,c$, and the numbers of $(u,v)$-paths of odd length in $G_1,G_2$ are $b,d$, respectively.
By induction $f(H_i)\geq 2t(H_i)-2$ for each $i\in \{1,2\}$.
Then $G_1$ has $f(H_1)-a$ odd cycles and $G_2$ has $f(H_2)-d$ odd cycles.
In total $G$ has at least $(f(H_1)-a)+(f(H_2)-d)+ad+bc$ odd cycles.
We know $a,d,b,c\geq 1$. So $ad+bc-a-d\ge (a-1)(d-1)+bc-1\geq 0$.
Thus $f(G)\geq f(H_1)+f(H_2)\geq (2t(H_2)-2)+(2t(H_2)-2)=2t(G)-2$.

By Theorem \ref{Thm:KY}, we have $f(G)\geq 2t(G)-2=2e(G)-2n\geq \frac{4}{3}(n-1)\geq n$,
with equality if and only if $n=4$ and $G=K_4$.
This completes the proof of Theorem \ref{Thm:4-cri-weak}.
\qed

\section{Counting cycles with parity via ear-decompositions}\label{sec:ear-decom}
In this section we prove two lemmas for counting cycles of specified parities
passing through a given vertex or a given edge in 3-connected non-bipartite (signed) graphs.
The key idea is to choose some ear-decomposition with particular properties,
based on a prefixed non-separating induced odd cycle.

\begin{lem}\label{lem:EarDecom-Fix-Vertex}
Let $G$ be a 3-connected non-bipartite signed graph,
$x$ be a vertex in $G$, and $D$ be a non-separating induced odd cycle in $G$ such that $x\notin V(D)$.
Let $R_i$ for $i\in [3]$ be three disjoint paths from $x$ to $z_i\in V(D)$ with $xy_i\in E(R_i)$.

Suppose there exists an edge-coloring $f$ assigning colors to every edge incident to $x$ such that $f(xy_i)$ for $i\in [3]$ are distinct.
Then $G$ contains at least $t(G)$ cycles of each parity passing through $x$ such that
the two edges incident to $x$ in every such cycle have different colors assigned by $f$.
\end{lem}

\begin{proof}
Let $t=t(G)$. We claim that there is an ear-decomposition $p_1\cup p_2\cup ...\cup p_t$ of $G$
such that $p_1=D, p_2=R_1\cup R_2, p_3=R_3$ and for each $i\geq 3$,
at least one of the ends of $p_i$ is not in $D$ and thus $D$ is non-separating in $G_i:=\cup_{j=1}^i p_j$.
To see this, suppose we already get desired ears $\{p_j\}_{1\leq j\leq i-1}$ for some $4\leq i\leq t$;
since $D$ is induced and non-separating in $G$, one can always find a new ear $p_i$ (a single edge or not) internally disjoint from $G_{i-1}$ with one end not in $D$.
For $i\geq 4$, let the ends of $p_i$ be $u_i, v_i$ with $v_i\notin V(D)$.
Since $D$ is non-separating in $G_{i-1}$, there is a path $L$ in $G_{i-1}-D$ from $v_i$ to $w\in V(R_1\cup R_2\cup R_3)-V(D)$.
As $G_{i-1}$ is 2-connected, there are two disjoint paths $L_1,L_2$ in $G_{i-1}$ from $\{v_i,u_i\}$ to $D\cup R_1\cup R_2\cup R_3$.
By concatenating with the path $L$ and renaming if necessary, we may assume that the end of $L_1$ other than $\{u_i,v_i\}$ is the vertex $w$ defined above.
Now we see that for each $i\geq 4$, there exists a path $Q_i:=p_i\cup L_1\cup L_2$ in $G_i$ containing the ear $p_i$ and internally disjoint from $D\cup R_1\cup R_2\cup R_3$,
where both ends are in $D\cup R_1\cup R_2\cup R_3$ but at most one is in $D$.

We observe that it will suffice to extend $Q_i$ to a path $Q_i'$ in $G_i$ with both ends in $D$ passing through $x$ such that its two edges incident to $x$ have different colors assigned by $f$.
Indeed, if true, then since $D$ is odd, by adding one of the two paths between two ends of $Q_i'$ in $D$ to $Q_i'$,
we can get a desired cycle of each parity for every $4\leq i\leq t$.
Since $p_i\subseteq Q_i'\subseteq G_i$, this provides $t-3$ distinct such cycles.
Also $D\cup R_1\cup R_2\cup R_3$ contains three desired cycles of each parity, so the lemma follows.

Finally we show how to extend $Q_i$ to $Q_i'$ in $G_i$.
This can be verified by considering all possible locations of the ends $w,w'$ of $Q_i$ in $D\cup R_1\cup R_2\cup R_3$.
Note that at least one of $w,w'$ is not in $V(D)$.
In case that $w,w'\in V(D\cup R_1\cup R_2\cup R_3)-x$, we omit the straightforward clarifications.
So it remains to consider when $x\in \{w,w'\}$ (say $x=w'$).
Let $xy\in E(Q_i)$ and by symmetry, $w\notin V(R_1\cup R_2)$.
There exists some $j\in [2]$ such that $f(xy_j)\neq f(xy)$.
If $w\in V(D)$, then $Q_i'$ can be chosen as $Q_i\cup R_j$;
otherwise $w\in V(R_3)$, then $Q_i'$ can be chosen as $z_3R_3w\cup Q_i\cup R_j$.
This completes the proof.
\end{proof}

\begin{lem}\label{lem:EarDecom-Fix-Edge}
Let $x, y$ be two distinct vertices in a 3-connected graph $G$ such that both $G-x$ and $G-y$ are non-bipartite.
Then $G$ contains at least $t(G)-1$ distinct $(x,y)$-paths of each parity (not including the possible edge $xy$).
\end{lem}

\begin{proof}
Let $H$ be obtained from $G$ by adding the edge $xy$ and let $t=t(H)$.
Then $H$ inherits all propositions of $G$ with $t(G)\leq t\leq t(G)+1$.

First we consider that $H-\{x,y\}$ is bipartite.
By Theorem \ref{thm:KT} (or Lemma \ref{lem_nonseparate}), we see that there exists a non-separating induced odd cycle $D$ in $H$ with $x\notin V(D)$.
Since $H-\{x,y\}$ is bipartite, such $D$ must contain $y$.
There exist two disjoint paths $P_1,P_2$ from $x$ to $D$ in $H-y$ internally disjoint from $D$.
Let $H'$ be obtained from $H$ by deleting all edges incident to $y$ except the two edges (say $yu,yv$) in $D$.
So $H'$ is 2-connected and $D$ is still non-separating in $H'$.
We can find an ear-decomposition $p_1\cup...\cup p_m$ in $H'$ such that
$p_1=D$, $p_2=P_1\cup P_2$ and for each $i\geq 3$, at least one end of $p_i$ is not in $D$, where $m=t(H')$.
So for $i\geq 3$, $D$ is non-separating in $H_i:=\cup_{i=1}^i p_j$.
By similar analysis as before, there exists a path $Q_i$ in $H_i$ containing the ear $p_i$ from $x$ to some vertex in $D-y$,
which can be extended to an $(x,y)$-path of each parity in $H_i$ containing $p_i$ for each $i\geq 3$.
Adding two such paths in $p_1\cup p_2$, we get $m$ desired $(x,y)$-paths in $H'$.
Also by Theorem \ref{thm:KT}, there exists a non-separating induced odd cycle $D'$ in $H$ with $x\in V(D')$ and $y\notin V(D')$.
Note that there are at least $t(G)-m-1$ edges $yz$ in $E(H)-E(H')$ for $z\notin \{u,v,x\}$.
We claim that for each such edge $yz$, there exists a path in $H$ from $y$ to some vertex in $D'-x$ which uses $yz$.
This is clear if $z\in V(D')$; for $z\notin V(D')$, since $H$ is 3-connected,
there exists a path in $H-\{x,y\}$ from $z$ to $D'-x$, from which the claim holds.
Using this claim, it is easy to find at least $t(G)-m-1$ many $(x,y)$-paths in $G$ of each parity,
which are also distinct from the above $m$ paths.
This finishes the proof when $H-\{x,y\}$ is bipartite.

Now we may assume that $H-\{x,y\}$ contains an odd cycle.
By Theorem \ref{thm:KT} there exists a non-separating induced odd cycle $D$ in $H$ such that $H-D$ contains $xy$.
We claim that there are four paths $P_1,P_2,P_3,P_4$ in $H$ from $\{x,y\}$ to $D$ such that
\begin{itemize}
\item [(a).] $x$ is an end of $P_1, P_2$ and $y$ is an end of $P_3, P_4$,
\item [(b).] any $P_i, P_j$ are internally disjoint, with at most one exception that $\{i,j\}=\{2,4\}$, and
\item [(c).] if $P_2$ and $P_4$ intersect, then $P_2=P_2'\cup R$ and $P_4=P_4'\cup R$ such that $P_2',P_4',R$ are internally disjoint paths and $x,y\notin V(R)$.
\end{itemize}
To prove this, since $H$ is 3-connected, we begin by choosing three internally disjoint paths $P_1,P_2,R$
in $H$ from $x,x,y$ to $a,b,c\in V(D)$, respectively.
There are also two disjoint paths $P_3, P_4$ in $H-x$ from $y$ to $D\cup P_1\cup P_2-x$, which are internally disjoint from $D\cup P_1\cup P_2$.
By concatenating $P_3, P_4$ with the path $R$ and renaming if necessary,
we may assume that $P_3$ is from $y$ to $c\in V(D)$ and by symmetry (between $P_1$ and $P_2$), $P_4$ is from $y$ to $D\cup P_2$ .
This proves the claim.

Next we build an ear-decomposition $p_1\cup ...\cup p_t$ of $H$ such that
$p_1=D, p_2=P_1\cup P_2, p_3=P_3\cup P_4$ (in case $P_2$ and $P_4$ interest, let $p_3=P_3\cup P_4'$), $p_4=xy$,
and for each $i\geq 5$, at least one end of $p_i$ is not in $D$ and $x, y$ cannot be the two ends of $p_i$.
The construction is similar as in the previous lemma (following the facts that $D$ is induced and non-separating in $H$ and $\{x,y\}$ is not a 2-cut of $H$),
and we omit the details here.
Let $H_i:=\cup_{j=1}^i p_j$ and $A$ be the vertex set of $p_1\cup ...\cup p_4$.

For fixed $i\geq 5$, let the ends of $p_i$ be $u, v$ with $v\notin V(D)$.
Since $H_{i-1}$ is 2-connected, $D$ is non-separating in $H_{i-1}$ and $\{x,y\}$ is not a 2-cut in $H_{i-1}$,
there exist two disjoint paths $L_1, L_2$ in $H_{i-1}$ from $\{u,v\}$ to $\{w_1, w_2\}\subseteq A$ and internally disjoint from $A$
such that $w_1\notin V(D)$ and $\{w_1,w_2\}\neq \{x,y\}$.
So $Q_i=p_i\cup L_1\cup L_2$ is a $(w_1,w_2)$-path in $H_i$ containing the ear $p_i$.
By distinguishing between all possible locations of $w_1,w_2$ in $A$,
it can be verified that there exist two disjoint paths $X_i, Y_i$ in $H_i$
from $x,y$ to two distinct vertices in $D$ such that $Q_i\subseteq X_i\cup Y_i$.
Since $D$ is odd, this provides an $(x,y)$-path of each parity in $H_i$ containing $p_i$ for every $5\leq i\leq t$.
So we get $t-4$ desired paths.
Also observing that $p_1\cup p_2\cup p_3$ contains at least three $(x,y)$-paths of each parity (not including the edge $xy$),
we see that $G$ has at least $t-1\geq t(G)-1$ desired $(x,y)$-paths.
This completes the proof.
\end{proof}

We remark that in Lemma \ref{lem:EarDecom-Fix-Edge}, if $xy$ is an edge then $G$ contains at least $t(G)-1$ distinct cycles of each parity passing through $xy$.

\section{Proof of Theorem \ref{thm:G=3-con}}

\begin{thm}\label{thm:3-con-bound-maxdeg}
Let $G$ be a 3-connected non-bipartite signed graphs with maximum degree at most $0.2t(G)$.
Then $f(G)\geq 0.02 t^2(G)$.
\end{thm}

\begin{proof}
Throughout this proof, let $T=t(G)$ and $\dG_T$ be the family of all 3-connected non-bipartite signed graphs with maximum degree at most $0.2T$.
So $G\in \dG_T$. We will show $f(G)\ge 0.02T^2$.
%For a signed graph $H$, let $\dC(H)$ be the collection of all odd cycles in $H$.
Our plan is to construct a sequence of signed graphs $G_0, G_1,...,G_q$ with the following properties:
\begin{itemize}
	\item [(\romannumeral1).] $G_i\in \dG_T$ for each $i\geq 0$, where $G_0=G$, and
%	\item [(\romannumeral2).] For each $i\geq 1$, there exists an injection $\phi_i: \dC(G_i)\to \dC(G_{i-1})$, and
	\item [(\romannumeral2).] For each $i\geq 1$, $f(G_{i-1})-f(G_{i })\geq \frac{1}{2}T\cdot (T_{i-1}-T_i)$ and $1\leq T_{i-1}-T_i\leq 0.4T$, where $T_i=t(G_i)$.
\end{itemize}
We will recursively define $G_i$ based on $G_{i-1}$ (the details will be given below),
and this process will terminate whenever the new $G_i$ satisfies either $T_i<0.8T$ or $T_i\geq 0.8T$ and $f(G_i)\geq 0.02T_i^2$.

Before defining these $G_i$'s, let us show how this desired sequence implies the conclusion.
If this process terminates at $G_q$ when $T_q\geq 0.8T$ and $f(G_q)\geq 0.02T_q^2$,
then by (\romannumeral2) we have
$$f(G)=f(G_q)+\sum_{i=1}^q (f(G_{i-1})-f(G_i))\geq 0.02T_q^2+\frac{1}{2}T\cdot (T-T_q)\geq 0.02 T^2.$$
Otherwise it terminates when $T_q<0.8T$, then by (\romannumeral2) we can also get $f(G)\geq \frac{1}{2}T\cdot (T-T_q)\geq 0.02 T^2$.

Now suppose for some $s\geq 0$, we have defined $G_i$'s for all $0\leq i\leq s$ as required.
We may assume
\begin{equation}\label{equ:Gs}
T_s\geq 0.8T \text{~~ and ~~} f(G_s)<0.02 T_s^2.
\end{equation}
In the rest of the proof, as we demonstrate, it suffices to define $G_{s+1}$ satisfying (\romannumeral1) and (\romannumeral2).
In steps to construct $G_{s+1}$, we will define several intermediate signed (multi-)graphs $M_\ell$ for $0\leq \ell\leq 3$.\footnote{For a multi-graph $M$, its {\it underlying graph} is a simple graph obtained from $M$ by deleting certain edges
so that only one edge of each adjacent pair of vertices remains. We say $M$ is $k$-connected (or bipartite) if and only if its underlying graph is so.
For a signed multi-graph $M$, let $f(M)$ be the number of all distinct odd cycles (of length at least three) in $M$.}

First we construct $M_0$ from $G_s$ as following.
Since $G_s\in \dG_T$, by Lemma \ref{lem_nonseparate} there exists a non-separating induced odd cycle $C$ in $G_s$.
If $|E(C,G_s-C)|\geq 4$, we simply define $M_0=G_s$.
Now consider $|E(C,G_s-C)|=3$. As $G_s$ is 3-connected and $C$ is induced,
we see that $C$ is a triangle say $xyzx$ and $E(C,G_s-C)$ consists of three independent edges say $xa,yb,zc$.
Now let $M_0$ be obtained from $G_s$ by deleting the vertex $z$, adding two new edges $xc, yc$, and assigning the parities of $xzc, yzc$ of $G_s$ to $xc, yc$, respectively.
In this case we will also rename $C$ by $xycx$ in $M_0$.

\begin{claim}\label{claim:M0}
$M_0$ is a 3-connected non-bipartite signed graph with maximum degree at most $0.2T+1$
and there exists a non-separating induced odd cycle $C$ in $M_0$ such that $|E_{M_0}(C,M_0-C)|\geq 4$, $t(M_0)=T_s$ and $f(G_s)\geq f(M_0)$.
Moreover, the only possible vertices of degree $0.2T+1$ belong to $C$.
\end{claim}
\begin{proof}
This is clear when $M_0=G_s$.
By the definition of $M_0$, we may assume that there exists an odd cycle $xyzx$ in $G_s$ and $E^*=E(xyz,G_s-xyz)$ consists of three independent edges $xa,yb,zc$.
By \eqref{equ:Gs}, $G_s\neq K_4$.
If $G_s-xyz$ is not 2-connected, then $G_s-xyz$ either is an edge or has at least two end-blocks;
in either case, it implies at least four edges in $E^*$, a contradiction.
So $G_s-xyz$ is 2-connected.
Now we see that the cycle $C=xycx$ is a non-separating induced odd cycle in $M_0$ with $|E(C,M_0-C)|\geq 4$ (where the oddness follows by the parities of $xc,yc$).
It is also easy to see that $M_0$ is 3-connected and non-bipartite with maximum degree at most $0.2T+1$ and $t(M_0)=t(G_s)=T_s$,
where the only vertex possibly having degree $0.2T+1$ is the vertex $c\in V(C)$.

So it remains to show $f(G_s)\geq f(M_0)$.
We prove this by showing an injection from odd cycles in $M_0$ to odd cycles in $G_s$.
Let $D$ be any odd cycle in $M_0$.
If $D$ contains none of $xc, yc$, then clearly $D$ is also an odd cycle in $G_s$.
If $D$ only contains one of $xc, yc$ (say $xc$),
then replacing $xc$ with $xzc$ in $D$ gives an odd cycle in $G_s$.
Lastly $D$ contains both $xc, yc$.
Since the parity of $xcy$ is the same as the parity of $xzy$,
replacing $xcy$ with $xzy$ in $D$ gives an odd cycle in $G_s$. This proves the claim.
\end{proof}

Adapting notations from Section \ref{sec:3-conn}, let $H=M_0-C$, $t=t(H)$ and $m=|E_{M_0}(C,H)|$.
By Claim \ref{claim:M0}, $T_s=t(M_0)=t+m$ and $m\geq 4$.
Using \eqref{equ:Gs} and $\Delta(M_0)\leq 0.2T+1$, we also can prove the following.

\begin{claim}\label{claim:t>0.6T}
Either $f(M_0)\geq 0.02T^2$, or $m\leq 0.2T$ and $t\geq 0.6T$. In the latter case, we have $M_0\in \dG_T$.
\end{claim}
\begin{proof}
First we show $f(M_0)\geq mt/2$. This holds trivially when $|V(H)|\in \{1,2\}$ (as we have $t=0$).
So $|V(H)|\geq 3$.
If $H$ is 2-connected, then by Lemma \ref{lem: 2-connected H} we get $f(M_0)\geq (t+1)m\geq mt/2$.
So we may assume that $H$ has $k\geq 2$ end-blocks.
Then Lemma \ref{lem:not 2 connected} shows that $f(M_0)\geq(m-k)(t+k)\geq mt/2$, where the last inequality holds because $m\geq 2k$ and thus $m-k\geq m/2$.
This proves $f(M_0)\geq mt/2$.

Let $C=x_1x_2...x_\ell x_1$ and $d_j=|N_H(x_j)|$.
For any two edges $x_ia_i, x_ja_j\in E(C,H)$ with $x_i\neq x_j$,
one can find an $(a_i,a_j)$-path in $H$.
Since $C$ is odd, together with one of the two $(x_i,x_j)$-paths in $C$, this provides an odd cycle\footnote{Recall that such odd cycle is called {\it basic} in Section \ref{sec:3-conn}.} in $M_0$.
Thus $f(M_0)\geq \sum_{i\neq j}d_id_j$.
If $m>0.6T$, since $\Delta(M_0)\leq 0.2T+1$ it is easy to divide $V(C)$ into two sets $X,Y$ such that $\sum_{x_i\in X}d_i\geq 0.2T$ and $\sum_{x_j\in Y}d_j\geq 0.2T$.
Then by Claim \ref{claim:M0}, $f(M_0)\geq(\sum_{x_i\in X}d_i)(\sum_{x_j\in Y}d_j)\geq 0.02 T^2$, completing the proof.
So we have $m\leq 0.6T$.
By \eqref{equ:Gs}, we get $t=T_s-m\geq T_s-0.6T\geq 0.2T$.
Since $0.02T_s^2>f(G_s)\geq f(M_0)\geq mt/2$, it follows that $m\leq \frac{0.04T_s^2}{0.2T}\leq 0.2T$ and then $t\geq T_s-m\geq 0.6T$.
Any vertex in $C$ has degree at most $m\leq 0.2T$ and thus by Claim \ref{claim:M0} we have $\Delta(M_0)\leq 0.2T$.
So $M_0\in \dG_T$. This proves Claim \ref{claim:t>0.6T}.
\end{proof}

Note that $f(G)\geq f(G_s)\geq f(M_0)$. So we may assume that the latter case of Claim \ref{claim:t>0.6T} holds.

Let $\dB$ be the set of all blocks in $H$ and $t_i=t(B_i)$ for each $B_i\in \dB$.
Let $\dT$ be a fixed spanning tree in $H$.
So the restriction of $\dT$ on any block of $H$ is also a tree.
For $a,b\in V(H)$, the unique subpath $a\dT b$ is called the {\it $(a,b)$-skeleton},
while any other $(a,b)$-path in $H$ is called a {\it non-skeleton}.

\begin{claim}\label{claim:B1}
There exists a unique 2-connected block $B_1$ in $H$ with $t_1=t(B_1)>T/2$ and $t-t_1<0.1T$.
\end{claim}
\begin{proof}
This is clear if $H$ is 2-connected by Claim \ref{claim:t>0.6T}. So $H$ is not 2-connected.
For any $B_i, B_j\in \dB$, there exists a path $P$ in the block structure of $H$ between two end-blocks say $D_1, D_2$ in $H$ and passing through $D_1, B_i, B_j, D_2$ in order.
Let the unique cut-vertex of $H$ contained in $D_\ell$ be $c_\ell$ for $\ell\in [2]$,
and let the two cut-vertices of $H$ incident to $B_i$ (respectively, to $B_j$) in $P$ be $\alpha_i, \beta_i$ (respectively, $\alpha_j, \beta_j$).
Since $M_0$ is 3-connected, one can easily find two independent edges $x_\ell y_\ell\in E(C,H)$ with $x_\ell\in V(C)$ and $y_\ell\in V(D_\ell)-c_\ell$ for $\ell\in [2]$.
By Lemma \ref{lem path}, for each $\ell\in \{i,j\}$ there exist $t_\ell$ non-skeleton $(\alpha_\ell,\beta_\ell)$-paths in $B_\ell$.
Using these non-skeletons, plus the $(y_1,\alpha_i)$-, $(\beta_i,\alpha_j)$- and $(\beta_j,y_2)$-skeletons,
one can find $t_it_j$ distinct $(y_1,y_2)$-paths in $H$, each of which yields a basic cycle.
So $f(G)\geq f(G_s)\geq f(M_0)\geq \sum_{B_i,B_j\in \dB} t_it_j$.
By Proposition \ref{prop:t}, $t=\sum_{B_i\in \dB} t_i\geq 0.6T$.
Let $t_1$ be the maximum over $t_i$'s.
If $t_1<0.2T$, then $\{t_i\}$ can be divided into two sets each of which has sum at least $0.2T$, implying that $f(G)\geq 0.04T^2$.
So $t_1\geq 0.2T$.
If $t-t_1\geq 0.1T$, then again $f(G)\geq t_1(t-t_1)\geq 0.02T^2$.
This shows $t_1>t-0.1T\geq 0.5T$, proving the claim.
\end{proof}

Next, we define $M_1$ to be obtained from the signed subgraph $M_0[B_1\cup C]$ by adding a new edge $xb$ for every $xa\in E_{M_0}(C,H-B_1)$ with $x\in V(C)$,
where $b\in V(B_1)$ be the unique cut-vertex separating $a$ and $B_1$ in $H$.
Moreover for every such new edge $xb$, we denote $P_{xb}:=xa\cup a\dT b$ and assign the parity of $xb$ to be the parity of $P_{xb}$.
We point out that $M_1$ is a multi-graph.

\begin{claim}\label{claim:M1}
$M_1$ is a 3-connected non-bipartite signed multi-graph such that $t(M_0)-t(M_1)=t-t_1$ and $f(M_0)-f(M_1)\geq t_1(t-t_1)$.
\end{claim}
\begin{proof}
Since $M_0$ is 3-connected, it is easy to verify that $M_1$ is 3-connected.
By the definition of $M_1$, we have $|E_{M_1}(B_1,C)|=|E_{M_0}(H,C)|$,
which together with Proposition \ref{prop:t} imply that $t(M_0)-t(M_1)=t-t_1$.
We now show that there exists an injection from odd cycles in $M_1$ to odd cycles in $M_0$.
Consider any odd cycle $D$ in $M_1$.
If $D$ does not contain any new edge in $M_1$, then obviously it is an odd cycle in $M_0$.
Suppose $D$ contains new edges in $M_1$.
For a new edge $xb$ which is not incident to any other new edges in $D$, then we can replace $xb$ by the path $P_{xb}$.
If there exists a pair of new edges $xb,yb$ in $D$ with $x,y\in V(C)$ and $b\in V(B_1)$,
then we can replace $xby$ by the symmetric difference of the paths $P_{xb}$ and $P_{yb}$,
which is an $(x,y)$-path in $M_0$ internally disjoint from $V(D)$ and has the same parity as $xby$ in $M_1$.
In this way, using the skeletons in $H$ we obtain a unique odd cycle in $M_0$ from $D$.
This gives the injection $\phi$ from odd cycles in $M_1$ to odd cycles in $M_0$.

Next we show that there are at least $t_1(t-t_1)$ odd cycles in $M_0$ which are distinct from the image of $\phi$.
Indeed, for any block $B_i\in \dB$ with $i\neq 1$, the proof of Claim \ref{claim:B1} provides at least $t_1t_i$ odd cycles in $M_0$
which use non-skeleton paths in $B_1, B_i$ and skeleton paths in other blocks.
Summing over all such blocks $B_i$, we prove that $f(M_0)-f(M_1)\geq t_1(t-t_1)$.
This finishes the proof of Claim \ref{claim:M1}.
\end{proof}

Let $M_2$ be obtained from $M_1$ by contracting the cycle $C$ into a new vertex $x^*$ and keeping all resulting multi-edges.
Given a partition $V(C)=X\cup Y$, let $M_{X,Y}$ be obtained from $M_1$ by contracting $X,Y$ into vertices $x,y$, respectively,
adding one edge $xy$ with parity 1 and keeping all other resulting multi-edges.
Since $C$ is induced, it is easy to see that $t(M_2)=t(M_{X,Y})=t(M_1)-1$.

\begin{claim}\label{claim:M2}
$M_2$ is 3-connected and there exists some $V(C)=X\cup Y$ such that $M_{X,Y}$ is 3-connected.
\end{claim}
\begin{proof}
Suppose that $M_2$ has a 2-cut $\{u,v\}$.
Since $M_1$ is 3-connected, the only possibility is $x^*\in \{u,v\}$, but this contradicts the 2-connectivity of $B_1$.
So $M_2$ is 3-connected.

Next we show that $M_{X,Y}$ is 3-connected if both $x$ and $y$ have at least two distinct neighbors in $B_1$.
Suppose there is a 2-cut $\{u,v\}$ in such $M_{X,Y}$.
Similarly the only possibility (by symmetry) is that $u\in V(B_1)$ and $v=x$.
Since $B_1-u$ is connected, it implies that $y$ has no neighbor in $B_1-u$.
That is, all neighbors of $y$ belong to $\{u,x\}$, a contradiction.

It suffices to show that there exists some $V(C)=X\cup Y$ such that in $M_{X,Y}$ both $x$ and $y$ have at least two distinct neighbors in $B_1$.
If $H$ is not 2-connected, then as in the explanation after Lemma \ref{lem:block tree},
one can define two staple edges for each end-block of $H$ in $M_0$ and thus $H$ has at least four such edges.
Using these four edges and by the definition of $M_1$, it is easy to find such a partition $X\cup Y$ of $V(C)$.
Thus $H$ is 2-connected. So $B_1=H$ and $M_1=M_0$.
By Claim \ref{claim:M0}, we have $|E_{M_1}(C,B_1)|\geq 4$.
In this case, again it is easy to find a desired partition $V(C)=X\cup Y$. This proves Claim \ref{claim:M2}
\end{proof}

Let $M_3$ be a signed multi-graph as following.
If $M_2$ is non-bipartite, then let $M_3=M_2$;
otherwise let $M_3$ be some 3-connected $M_{X,Y}$ guaranteed by Claim \ref{claim:M2}.
By the definition we see that $M_3$ is 3-connected with $t(M_3)=t(M_1)-1$.
Next we show that $M_3$ is also non-bipartite.
It is enough to consider when $M_3=M_{X,Y}$.
In this case, $M_2$ is bipartite, so any cycle in $M_2$ passing through $x^*$ is even.
This also implies that any $(x,y)$-path in $M_3=M_{X,Y}$ (except the edge $xy$) is even.
Since the parity of $xy$ in $M_3$ is one, we see that indeed $M_3$ is non-bipartite.

Finally, we define $G_{s+1}$ to be a underlying graph of $M_3$
(that is, to keep only one edge of each adjacent pair of vertices in $G_{s+1}$) such that it contains at least one odd cycle.
Let $\alpha=t(M_3)-t(G_{s+1})$, which is the number of deleted edges in this process.
Clearly each of the deleted edges corresponds to one in $E_{M_1}(C,B_1)$.
So by Claim \ref{claim:M0}, we have $\alpha\leq m\leq 0.2T$.

\begin{claim}\label{claim:Gs+1}
$G_{s+1}$ is a 3-connected non-bipartite signed graph such that $t(M_1)-t(G_{s+1})=\alpha+1$ and $f(M_1)-f(G_{s+1})\geq t_1(\alpha+1)$.
\end{claim}

\begin{proof}
By definition, it is clear that $G_{s+1}$ is a 3-connected and non-bipartite signed graph such that $t(M_1)-t(G_{s+1})=\alpha+1$ and $t(G_{s+1})\geq t(B_1)=t_1$.

To show $f(M_1)-f(G_{s+1})\geq t_1(\alpha+1)$, we first give an injection $\phi$ from odd cycles in $G_{s+1}$ to odd cycles in $M_1$.
Let $Q$ be any odd cycle in $G_{s+1}$.
In the case $M_3=M_2$, if $x^*\notin V(Q)$, then $Q$ is also an odd cycle in $M_1$;
otherwise $x^*\in V(Q)$, then the two edges in $Q$ incident to $x^*$ have the same end in $C$ or different ones (say $u,v$).
In the former case, $Q$ also corresponds to an odd cycle in $M_1$;
in the latter case, adding the even $(u,v)$-path in $C$ to the preimage of $Q$ in $M_1$ gives a unique odd cycle in $M_1$.
Now consider the case $M_3=M_{X,Y}$.
Since $M_2$ is bipartite, all $(x,y)$-paths in $M_{X,Y}$ (except the edge $xy$) are even and any odd cycle $Q$ in $G_{s+1}$ must use $x$ and $y$.
In fact such $Q$ must use $xy$ (as otherwise one of the two $(x,y)$-paths in $Q$ is odd, a contradiction).
Then again adding one of two paths in $C$ between the ends of the preimage of $Q$ gives a unique odd cycle in $M_1$.
This defines the injection $\phi$.

We now show that there are at least $t_1(\alpha+1)$ odd cycles in $M_1$, which are distinct from the image of $\phi$.
First we consider any edge $e\in E(M_3)\backslash E(G_{s+1})$, which corresponds to an edge $uv$ in $E_{M_1}(C,B_1)$ with $u\in V(C)$.
Since $M_1$ is 3-connected, there exists an edge $u'v'$ in $E(M_1)$ with $u'\in V(C)-u$ and $v'\in V(B_1)-v$.
We can choose $u'v'$ so that it corresponds to an edge in $G_{s+1}$.
Since $B_1$ is 2-connected, by Lemma \ref{lem path} there are at least $t_1$ distinct $(v,v')$-paths in $B_1$.
Adding the edges $uv, u'v'$ and one of the two $(u,u')$-paths in $C$ to the each of these paths gives an odd cycles in $M_1$.
There are $\alpha$ such edges $e$, which provides at least $t_1\alpha$ distinct odd cycles in $M_1$.
Clearly these odd cycles are also distinct from the image of $\phi$.

It remains to show there are other $t_1$ odd cycles in $M_1$ which are distinct from the above ones.
We will prove this by distinguishing among the following three cases.

Suppose that the signed graph $B_1$ is non-bipartite. In this case $M_3=M_2$.
By Lemma \ref{lem_nonseparate}, there exists a non-separating induced odd cycles $D$ in $G_{s+1}$ such that $x^*\notin V(D)$.
Since $M_1$ is also 3-connected, there exist three disjoint paths from $D$ to $C$ in $M_1$,
which yields three internally disjoint paths $R_1,R_2,R_3$ from $D$ to $x^*$ in $G_{s+1}$.
To apply Lemma \ref{lem:EarDecom-Fix-Vertex}, we define an edge-coloring $f$,
which assigns every edge $x^*y$ in $G_{s+1}$ by the color $x_i\in V(C)$, where $x_iy$ is the preimage of $x^*y$ in $M_1$.
Clearly, the three edges of $R_1, R_2, R_3$ incident to $x^*$ have distinct colors assigned by this $f$.
By Lemma \ref{lem:EarDecom-Fix-Vertex} (with $G=G_{s+1}$),
$G_{s+1}$ contains at least $t(G_{s+1})\geq t_1$ even cycles passing through $x^*$ such that the two edges incident to $x^*$ in every such cycle have different colors assigned by $f$.
The preimage of every such cycle is an even path with two different ends in $C$.
Since $C$ is odd, adding the odd path of $C$ between the two ends to this preimage results in an odd cycle in $M_1$.
It is easy to see that these odd cycles are distinct from the odd cycles in $M_1$ found above.
So in this case $f(M_1)-f(G_{s+1})\geq t_1(\alpha+1)$.

Now suppose that $B_1$ is bipartite but $M_2$ is non-bipartite. Again in this case we have $M_3=M_2$.
By Proposition \ref{prop:bip}, there exists a bipartition $V(B_1)=I\cup J$
such that each $e\in E(I,J)$ is odd and each $e\in E(B_1)\backslash E(I,J)$ is even.
Since $M_1$ is 3-connected, there exist three independent edges say $x_ia_i$ in $E_{M_1}(C,B_1)$ with $x_i\in V(C)$ for $i\in [3]$,
which correspond to three edges $x^*a_i$ in $G_{s+1}$ for $i\in [3]$.
Then we can find two vertices say $a_1,a_2$ such that
either $x^*a_1,x^*a_2$ have the same parity and $a_1,a_2$ belong to the same part,
or $x^*a_1,x^*a_2$ have the opposite parity and $a_1,a_2$ belong to the different parts.
Since $B_1$ is 2-connected, by Lemma \ref{lem path} there are $t_1$ distinct $(a_1,a_2)$-paths in $B_1$.
By our choice, these paths give at least $t_1$ even cycles in $G_{s+1}$ passing through $x^*$ (by adding $x^*a_1, x^*a_2$)
and at least $t_1$ odd cycles in $M_1$ (by adding $x_1a_1, x_2a_2$ and the unique odd $(x_1,x_2)$-path of $C$).
This also proves $f(M_1)-f(G_{s+1})\geq t_1(\alpha+1)$.

Lastly we consider the case that $M_2$ is bipartite. Then $M_3=M_{X,Y}$.
As $M_1$ is 3-connected, there are three independent edges $x_ia_i$ in $E_{M_1}(C,B_1)$ for $i\in[3]$.
Now two of them are incident with one of $x,y$ (say they are $xa_1,xa_2\in E(G_{s+1})$).
By Lemma \ref{lem path} there are at least $t_1$ distinct $(a_1,a_2)$-paths in $B_1$.
Since $M_2$ is bipartite, adding $xa_1,xa_2$ to these paths result in at least $t_1$ even cycles in $G_{s+1}$ passing through $x$.
On the other hand, adding $x_1a_1,x_2a_2$ and the unique odd $(x_1,x_2)$-path in $C$ will give at least $t_1$ odd cycles in $M_1$,
which are distinct from the image of $\phi$ as well as these odd cycle raised from edges in $E(M_3)\backslash E(G_{s+1})$.
This completes the proof of Claim \ref{claim:Gs+1}.
\end{proof}

To conclude this proof, we now show that $G_{s+1}$ satisfies the propositions (\romannumeral1) and (\romannumeral2).
Let $T_{s+1}=t(G_{s+1})$.
Combining the claims \ref{claim:M0}, \ref{claim:M1} and \ref{claim:Gs+1},
we get $T_s-T_{s+1}=t-t_1+\alpha+1$ and $f(G_s)-f(G_{s+1})\geq t_1(T_s-T_{s+1})$.
By Claim \ref{claim:B1}, $t_1>T/2$ and $0\leq t-t_1<0.1T$.
Also we have $\alpha\leq m\leq 0.2T$.
Thus it follows that $1\leq T_s-T_{s+1}\leq 0.4T$ and $f(G_s)-f(G_{s+1})\geq \frac12 T\cdot (T_s-T_{s+1})$.
This proves (\romannumeral2).

To prove (\romannumeral1), it suffices to show that the maximum degree $\Delta(G_{s+1})$ is at most $0.2T$.
By Claim \ref{claim:t>0.6T}, $\Delta(M_0)\leq 0.2T$ and $m\leq 0.2T$.
So each of the new vertices $x^*, x, y$ has degree at most $m\leq 0.2T$ in $G_{s+1}$.
In the case $M_3=M_2$, suppose there exists some $u\in V(B_1)$ with $d_{G_{s+1}}(u)>|N_{M_0}(u)\cap (C\cup B_1)|$.
Then $u$ must be a cut-vertex in $H$ and $d_{G_{s+1}}(u)=|N_{M_0}(u)\cap (C\cup B_1)|+1\leq d_{M_0}(u)\leq 0.2T$.
This shows that $\Delta(G_{s+1})\leq 0.2T$ when $M_3=M_2$.
Now let us assume $M_3=M_{X,Y}$.
By the similar arguments as above, one can derive that $\Delta(G_{s+1})\leq 0.2T+1$ and
if $u\in V(G_{s+1})$ has degree $0.2T+1$ in $G_{s+1}$, then $u\in V(B_1)$ is adjacent to both $x$ and $y$.
Note that in this case $M_2$ is bipartite,
so the parity of the path $xuy$ is even.
Since the parity of $xy$ is 1 and $B_1$ is 2-connected,
the cycle $C'=xuyx$ is a non-separating induced odd cycle in $G_{s+1}$.
Applying Claim \ref{claim:t>0.6T} with $M_0=G_{s+1}$ (note that in the proof of this claim we also make sure of $\Delta(M_0)\leq 0.2T+1$),
either $f(G)\geq f(G_{s+1})\geq 0.02T^2$, or $d_{G_{s+1}}(u)\leq |E(C',G_{s+1}-C')|\leq 0.2T$ for every such $u$.
So we may assume that the latter case occurs and thus $\Delta(G_{s+1})\leq 0.2T$.
This finishes the proof of Theorem \ref{thm:3-con-bound-maxdeg}.
\end{proof}

Now we are ready to prove Theorem \ref{thm:G=3-con}.

\medskip

{\noindent \bf Proof of Theorem \ref{thm:G=3-con}.}
Let $G$ be a 3-connected non-bipartite graph.
If $\Delta(G)\leq 0.2t(G)$, then by Theorem \ref{thm:3-con-bound-maxdeg}, we have $f(G)\geq 0.02t^2(G)$.
So we may assume that there is a vertex $x$ of degree at least $0.2t(G)+1$.
Suppose there exists an odd cycle $C$ in $G\backslash x$.
For any distinct $a,b\in N(x)$, as $G\backslash x$ is 2-connected,
there are two disjoint paths from $\{a,b\}$ to $u,v\in V(C)$, which together with one of the two $(u,v)$-paths in $C$ give an odd $(a,b)$-path in $G\backslash x$.
Thus $f(G)\geq \binom{d(x)}{2}\geq 0.02t^2(G)$.

Now it is fair to assume that $G\backslash x$ is bipartite with parts $A,B$.
Let $T=t(G)$, $t=t(G\backslash x)$, $d_1=|N(x)\cap A|$ and $d_2=|N(x)\cap B|$.
Since $G$ is 3-connected and non-bipartite, $G[A\cup B]=G\backslash x$ is 2-connected and we may assume $d_1\geq d_2\geq 1$.
This implies that $d_1\geq d(x)/2\geq 0.1T$.
By Lemma \ref{lem path} there are at least $t+1$ paths in $G\backslash x$ between any vertex in $N(x)\cap A$ and any vertex in $N(x)\cap B$, all of which have odd lengths.
Thus $f(G)\geq d_1d_2(t+1)\geq d_1(d_2+t)$.
Note that we have $T+1=d_1+d_2+t$ and $d_1\geq 0.1T$.
If $d_2+t\geq d_1$, then $f(G)\geq d_1(d_2+t)\geq 0.09T^2$, as desired.
So we may assume that $d_1\geq d_2+t$.
By the same analysis, we may further assume that $d_2+t\leq 0.1T$ and $d_1\geq 0.9T$.

So $n-1\geq d(x)\geq d_1\geq 0.9T$.
Let $B_i$ be the set of vertices in $B$ of degree $i$ in $G\backslash x$ for $i\geq 2$.
Since $G$ is 3-connected, we have $d_2\geq |B_2|$ and $e(A,B)\geq 2|A|$.
Also $e(A,B)=\sum_{i\geq 2} i|B_i|$,
so $$2t\geq 2(e(A,B)-|A|-|B|)\geq e(A,B)-2|B|=\sum_{i\geq 2} i|B_i|-2\sum_{i\geq 2}|B_i|=\sum_{i\geq 3} (i-2)|B_i|.$$
Thus using $2|A|\leq e(A,B)=\sum_{i\geq 2} i|B_i|$, we get $2(|A|-|B|)\leq \sum_{i\geq 3} (i-2)|B_i|\leq 2t$.
Now we have
$$2d_2+4t\geq 2|B|=(|A|+|B|)-(|A|-|B|)\geq n-1-t\geq 0.9T-t,$$
which implies that $2d_2+5t\geq 0.9T$, a contradiction to $d_2+t\leq 0.1T$.
This proves Theorem \ref{thm:G=3-con}.
\qed

\section{Proof of Theorem \ref{Thm:4-cri-str}}
We prove this by induction on the number of vertices.
The base case $G=K_4$ is clear. Let $G$ be a 4-critical graph.
If $G$ is 3-connected, then this follows by Theorem \ref{thm:G=3-con}.
So there exists some 2-cut $\{x,y\}$ in $G$.
By Lemma \ref{lem:2-cut}, $xy\notin E(G)$ and there are unique proper induced subgraphs $G_1, G_2$ of $G$ such that $G=G_1\cup G_2$ and $V(G_1)\cap V(G_2)=\{u,v\}$.
We choose a 2-cut $\{x,y\}$ such that $G_1$ has the minimum order among all choices.
By the minimality we see that $G_1+xy$ is 3-connected.
By Lemma \ref{lem:2-cut} again either (1) $H_1:=G_1+xy$ and $H_2:=G_2/\{x,y\}$ are 4-critical or (2) $H_1:=G_1/\{x,y\}$ and $H_2:=G_2+xy$ are $4$-critical.
In either case, we have $t(H_i)=t(G_i)+1$ for each $i\in [2]$ and $t(G)+1=t(H_1)+t(H_2)$.
By induction, $f(H_i)\geq 0.02t^2(H_i)$ for each $i\in [2]$.

Suppose (1) occurs.
Fix an $(x,y)$-path $P_1$ in $G_1$ of even length.
Any odd cycle in $H_2$ becomes either an odd cycle or an odd $(x,y)$-path in $G_2$.
In the latter case, concatenating with $P_1$ gives an odd cycle in $G$.
So we get $0.02t^2(H_2)$ distinct odd cycles in $G$ from $H_2$.
Also fix an $(x,y)$-path $P_2$ in $G_2$ of odd length (such path is easy to see).
By similar augments, concatenating with $P_2$ if needed, we get $0.02t^2(H_1)$ odd cycles in $G$ from $H_1$.
Next we combine $(x,y)$-paths in $G_1$ and $G_2$ (but not using $P_1,P_2$) to get more odd cycles in $G$.
Since $G_1+xy$ is 3-connected and 4-critical,
by Lemma \ref{lem:EarDecom-Fix-Edge},
there are at least $t(G_1+xy)-1=t(G_1)$ distinct $(x,y)$-paths (except the edge $xy$) of each parity in $G_1+xy$ (thus in $G_1$).
By Lemma \ref{lem path}, since $G_2+xy$ is 2-connected,
there are at least $t(G_2+xy)=t(G_2)+1$ distinct $(x,y)$-paths (except the edge $xy$) in $G_2$.
Thus for every such path (except $P_2$) in $G_2$,
there are at least $t(G_1)-1$ distinct $(x,y)$-paths (excluding $P_1$) in $G_1$ of opposite parity.
This yields at least $t(G_2)(t(G_1)-1)$ odd cycles in $G$, all of which are distinct from the above ones derived from $H_1$ and $H_2$.
Summing up, we get
$$f(G)\geq 0.02t^2(H_1)+0.02t^2(H_2)+t(G_2)(t(G_1)-1)\geq 0.02t^2(G).$$
Now suppose (2) occurs. In this case $H_1=G_1/\{x,y\}$ is 4-critical.
So both $(G_1+xy)-x$ and $(G_1+xy)-y$ are non-bipartite.
Recall that $G_1+xy$ is 3-connected.
By Lemma \ref{lem:EarDecom-Fix-Edge}, there are at least $t(G_1+xy)-1=t(G_1)$ distinct $(x,y)$-paths (except the edge $xy$) of each parity in $G_1+xy$.
By similar analysis as above, we also can derive that
$f(G)\geq 0.02t^2(H_1)+0.02t^2(H_2)+t(G_2)(t(G_1)-1)\geq 0.02t^2(G).$
This completes the proof of Theorem \ref{Thm:4-cri-str}.
\qed

\section{Concluding remarks}
In this paper we consider a problem of Gallai from 1984 which asks whether for $k\geq 4$
the number of distinct $(k-1)$-critical subgraphs in any $k$-critical graph is at least the order of the graph $n$.
For general $k$, we improve a longstanding lower bound on this number proved by Abbott and Zhou \cite{AZ} since 1995.
In the case $k=4$ -- the main focus of this paper, we show this number is at least $\Omega(n^2)$,
which is tight up to the constant factor by infinitely many 4-critical graphs.
In addition, we give a very short proof to Gallai's problem for $k=4$ (by a different approach from \cite{H19}).
Along the way to obtain these, we developed some tools for counting cycles with specified parity and passing through some fixed vertex or edge
(see Lemmas \ref{lem:EarDecom-Fix-Vertex} and \ref{lem:EarDecom-Fix-Edge});
a key ingredient in these lemmas is a novel application of the ear-decomposition together with the use of non-separating cycles.
For the needs of the approach, we also consider and establish the analogous results in signed graphs,
which may be of interest on its own.

In relation to the results provided here, besides the original problem of Gallai, there are many interesting problems one can ask for.
One may wonder if Theorem \ref{thm:G=3-con} also can be extended to the setting of signed graphs.
However, unlike Theorem \ref{thm:3-con-bound-maxdeg}, the following example shows in negative.

\begin{constr}
Assume that $(A,B)$ is a bipartition of an even cycle $C_{2n}$.
Let $H$ be obtained from this $C_{2n}$ by adding a vertex $x$ and edges $xu$ for all $u\in A\cup B$.
Fix a vertex $b\in B$.
Assign $0$ to edges $xu$ for all $u\in B-\{b\}$ and assign $1$ to all edges in $C_{2n}$ and edges $xu$ for all $u\in A\cup \{b\}$.
\end{constr}

\noindent It is not hard to see that $H$ is a 3-connected non-bipartite signed graph,
every odd cycle in $H$ passes through the edge $xb$ and thus $H$ contains at most $2t(H)$ odd cycles.
This also explains that it is needed to bound the maximum degree in Theorem \ref{thm:3-con-bound-maxdeg}.

In Theorem \ref{Thm:4-cri-str} we prove that $\min f_3(G)=\Theta(n^2)$, where the minimum is over all $n$-vertex 4-critical graphs $G$.
This oversteps the original linear bound asked by Gallai in the case $k=4$.
The following problem seems natural to ask.
\begin{prob}
Determine the order of the magnitude of $\min f_{k-1}(G)$ over all $n$-vertex $k$-critical graphs $G$ for all $k\geq 5$.
\end{prob}

\noindent It is of particular interest to consider the above minimum for all $n$-vertex 3-connected $k$-critical graphs.
We are not sure if the additional 3-connectivity condition will change the magnitude of the minimum for $k\geq 5$, which would also be interesting to know.
In the case of $k=4$, we know the additional 3-connectivity condition does not change much, as there are 4-critical $n$-vertex graphs in both cases (3-connected or not) with $O(n^2)$ distinct odd cycles.

Let $k\geq 4$. We would like to emphasise here that in this paper, all results on 4-critical graphs can be easily extended to $k$-critical graphs.
The reason is that the only structural property we used for 4-critical graphs is Lemma \ref{lem:2-cut}, which also holds for all $k$-critical graphs.
For instance, Theorem \ref{Thm:4-cri-str} can be restated as that any $n$-vertex $k$-critical graphs $G$ has at least $0.02t^2(G)\geq \Omega(n^2)$ distinct odd cycles.
We believe a better bound on the number of odd cycles should hold for $k\geq 5$.

\begin{prob}\label{prob:k-critical->odd cycle}
Determine the order of the magnitude of the minimum number of distinct odd cycles over all $n$-vertex $k$-critical graphs for all $k\geq 5$.
\end{prob}

\noindent It is easy to see that such number must be a polynomial function of $n$. 

Lastly we point out that the lemmas in Sections \ref{sec:3-conn} and \ref{sec:ear-decom} also can yield the same number of distinct even cycles in the circumstances therein.
Hence one can derive the following for even cycles.

\begin{thm}
Let $G$ be a graph which is either 4-critical or 3-connected. Then $G$ contains at least $\Omega(t^2(G))$ distinct even cycles.
\end{thm}

\noindent We give a sketch for its proof as follows. If such $G$ is bipartite, then it holds easily by a recursive use of Lemma \ref{lem path} in any ear-decomposition of $G$.
Otherwise $G$ is either 3-connected non-bipartite or 4-critical, then it follows by analogous proofs as in Theorems \ref{thm:G=3-con} and \ref{Thm:4-cri-str}.
This bound is also tight up to the constant factor, as indicated by (even and odd) wheels $W(n,1)$, which are 3-connected too.

One can ask for the analog of Problem \ref{prob:k-critical->odd cycle} for even cycles as well.
For more problems on $k$-critical graphs, we refer to the book \cite{JT} by Jensen and Toft.

\medskip

\bigskip

\noindent {\bf Acknowledgement.}
We would like to thank Asaf Shapira for providing counterexamples to some problems we asked in an earlier version of this paper.

\bibliographystyle{unsrt}

\end{document}